\documentclass[12pt]{amsart}
\usepackage{amsmath}
\usepackage{amsfonts}
\usepackage{amsthm}
\usepackage{amssymb}
\usepackage{amscd}
\usepackage[all]{xy}
\usepackage{enumerate}
\usepackage{mathrsfs}
\usepackage{enumitem}
\usepackage{graphicx}
\usepackage{float}
\usepackage{color}
\usepackage{hyperref}



\newcommand*{\ext}{\mathcal{E}\kern -.7pt xt}

\newcommand*{\hhom}{\mathcal{H}\kern -.7pt om}

\theoremstyle{plain}
\newtheorem{thm}{Theorem}[section]

\newtheorem*{mthm*}{Main Theorem}

\newtheorem{prop}[thm]{Proposition}

\newtheorem{cor}[thm]{Corollary}

\newtheorem{lem}[thm]{Lemma}

\theoremstyle{definition}
\newtheorem{defn}[thm]{Definition}

\newtheorem{expl}[thm]{Example}

\newtheorem*{ackn}{Acknowledgment}

\newtheorem{rmk}[thm]{Remark}

\newcommand{\sD}{\mathcal{D}}
\newcommand{\sE}{\mathcal{E}}
\newcommand{\sF}{\mathcal{F}}
\newcommand{\sG}{\mathcal{G}}

\newcommand{\sL}{\mathcal{L}}

\newcommand{\sM}{\mathcal{M}}
\newcommand{\sO}{\mathcal{O}}

\newcommand{\sX}{\mathcal{X}}

\newcommand{\mC}{\mathbb{C}}
\newcommand{\mD}{\mathbb{D}}

\newcommand{\mP}{\mathbb{P}}

\newcommand{\mR}{\mathbb{R}}

\newcommand{\Ker}{\mathrm{Ker}\,}

\newcommand{\rank}{\mathrm{rank}\,}

\numberwithin{equation}{section}

\newcommand{\beba}  {\begin{equation}\begin{array}{rcl}}

\newcommand{\eaee}  {\end{array}\end{equation}}

\makeatletter
\def\l@section{\@tocline{1}{0pt}{1pc}{}{}}
\def\l@subsection{\@tocline{2}{0pt}{1pc}{4.6em}{}}
\def\l@subsubsection{\@tocline{3}{0pt}{1pc}{7.6em}{}}
\renewcommand{\tocsection}[3]{%
  \indentlabel{\@ifnotempty{#2}{\makebox[2.3em][l]{%
    \ignorespaces#1 #2.\hfill}}}#3}
\renewcommand{\tocsubsection}[3]{%
  \indentlabel{\@ifnotempty{#2}{\hspace*{2.3em}\makebox[2.3em][l]{%
    \ignorespaces#1 #2.\hfill}}}#3}
\renewcommand{\tocsubsubsection}[3]{%
  \indentlabel{\@ifnotempty{#2}{\hspace*{4.6em}\makebox[3em][l]{%
    \ignorespaces#1 #2.\hfill}}}#3}
\makeatother

\title{On supported deformations and birational isotriviality}

\author{Luca Rizzi}
\address{Luca Rizzi\\Center for Geometry and Physics \\
	Institute for Basic Science (IBS)\\
	Pohang 37673\\ Korea,
	\texttt{lucarizzi@ibs.re.kr}}

\author{Francesco Zucconi}
\address{Francesco Zucconi\\Department of Mathematics, Computer Science and Physics \\
Universit\`a degli Studi di Udine\\
Udine, 33100\\ Italia
\texttt{francesco.zucconi@uniud.it}}

\begin{document}

\markboth{}{}

\begin{abstract}

It is well known that the general fibers of a fibration $f\colon X\to B$ are isomorphic if the general Kodaira-Spencer class vanishes. In this paper we consider the birational analogue when the general Kodaira-Spencer class is supported on a divisor.
\end{abstract}
\maketitle
\section{Introduction}
Given a proper holomorphic fibration over a curve $f\colon X\to B$, the link between the vanishing of the general Kodaira-Spencer class and the isomorphism among general fibers is well known, see \cite{KS,KM}. 

On the other hand, we prove that if the general fibers of $f$ are birational, then their Kodaira-Spencer classes are supported on a divisor, and this divisor is not movable. 
This result is possibly well-known, but to the best of our knowledge we did not find a reference for it: see Theorem  \ref{birimplicasupp}.

We recall the definition of supported deformation. Let $Z$ be a smooth complex projective variety, $T_Z$ its tangent sheaf and $D$ an effective divisor on $Z$. An infinitesimal deformation class $\zeta\in H^1(Z,T_Z)$ is said to be \textit{supported on} $D$ if $\zeta$ is in the kernel of the natural map $$H^1(Z,T_Z)\to H^1(Z,T_Z(D)).$$

When $Z$ is a curve, a deformation supported on a single point is known as Schiffer variation, see \cite{SS}. These are very well studied, cf. \cite{ACG}. In the case of arbitrary dimension see  \cite{victor,RZ1}.

 The main result of this paper gives conditions for the vice versa of the above mentioned Theorem \ref{birimplicasupp}, that is it proves that the general fibers are birational if the general Kodaira-Spencer class is supported on a suitable divisor.  Clearly there is a vast literature on deformations of the couple $(Z,D)$. Nevertheless, our main result fills a gap which, once identified, seems very natural to tackle.

In the classical case of a proper holomorphic submersion over a  disk $f\colon\sX\to\Delta$, the existence of a holomorphic vector field on $\sX$ that lifts the coordinate tangent vector field on $\Delta$ implies that  the family is locally trivial; see \cite{KS,KM}, see also \cite{M}. This trivialisation is induced by the flow of the vector field on $\sX$. We stress that in this argument the properness of $f\colon\sX\to\Delta$ plays a crucial role. Indeed, it is precisely the fact that the fibers are compact that ensures that the integral curves are defined over a disk of positive radius independent of their respective initial point on the central fiber.
 
As a first step for our main result, in Proposition \ref{vettorepoli} we show that when the Kodaira-Spencer classes are supported on an effective divisor, there exists a meromorphic vector field on $\sX$ with poles on the divisor and lifting the tangent vector field on $\Delta$.

 Now the flow argument cannot be immediately adapted to show that the fibers are isomorphic outside this divisor; the main obstacle is that we have  lost the properness condition. Intuitively, it can happen that the closer a point of the central fiber is to the poles of the vector field, the smaller the radius where the integral curve is defined. Hence  it may not be possible to define the flow over a disk of fixed positive radius.
 
 However, a classical result, called Volumetric Theorem, see \cite[Theorem 1.5.3]{PZ}, ensures birationality among the fibers if a certain volume constrain holds for $f\colon\sX\to\Delta$. This constrain is strongly related to supported deformations.
 The main idea of this paper is to combine geometrical arguments relying on the flow of the meromorphic vector field with a similar volume constrain. Indeed volume constrains control the flow of the meromorphic vector field.

More precisely, consider $f\colon X\to B$ a semistable fibration, see \cite{Il}, with fibers of general type between a smooth complex $n+1$ dimensional projective variety $X$ and a smooth complex projective curve $B$. Let $\mD^1$ be the local system on $B$  of relative 1-forms defined by the de Rham closed $1$-forms of $X$, and $U=\langle\eta_1,\dots,\eta_{k+1}\rangle$, $0<k\leq n$, be a  subspace generated by $k+1$-linearly independent monodromy invariant sections of $\mD^1$. We say that $U$ is a {\it{a volume detecting subspace}} if $\eta_1\wedge\dots\wedge\eta_{k+1}=0$ and if $U$ is fiberwise strict, see Definition \ref{volumedet} for all the details and also \cite{Ca2} for the notion of strictness of a vector space of differential 1-forms.

 We consider  a  horizontal divisor $\sD$ on $X$ contained in the common zeroes of the relative forms $\omega_i :=\eta_1\wedge\dots\wedge\hat{\eta_i}\wedge\dots\wedge\eta_{k+1}$, $i=1,...,k+1$, and we denote by $D_b$ the restriction of $\sD$ on the general fiber $X_b$. We say that $\sD$ is relatively non movable if $h^0(X_b,\sO_{X_b}(D_b))=1$ for a general point $b$ in $B$. We also say that the Kodaira-Spencer class  is generically fiberwise supported on $\sD$ if $\xi_b$ is supported on $D_b$, where $\xi_b$ is the Kodaira-Spencer class of $X_b$ and $b$ is a general point of $B$, see Definition \ref{supportato3} and Remark \ref{generalita}.
 
 We prove in Proposition \ref{esisteY} that a fibration $f\colon X\to B$ with a volume detecting subspace $U$ and  relatively non movable  divisor $\sD$  and with generically fiberwise supported Kodaira-Spencer class   has a surjective morphism $h\colon X\to Y$ over a $k$-dimensional normal variety $Y$. We denote by $\sF$ the foliation induced by the relative tangent bundle of $h$.
 
 When $\sD$ is invariant under $\sF$, see Definition \ref{invariante}, we can overcome the lack of properness mentioned above.
  We prove:

 \begin{mthm*}Let $f\colon X\to B$ be a semistable fibration with fibers of general type. Assume that there exist a volume detecting subspace $U$ with a relatively non movable divisor $\sD$. If the Kodaira-Spencer class is generically fiberwise supported on $\sD$ and $\sD$ is $\sF$-invariant then  the general fibers of $f\colon X\to B$ are birational.
 \end{mthm*}
 See Theorem \ref{main}.
 
The two conditions on $\sD$, that is $\sD$ is invariant under $\sF$ and it is relatively not movable, occur when the fibers are birational, see Remark \ref{necessarie} and \ref{necessarie2}. 
Indeed, these conditions are strictly necessary, as we show in the very explicit examples and constructions of Section \ref{sez6}. To sum up, they can be considered very natural hypotheses to obtain birationality among fibers of a fibration with generically fiberwise supported Kodaira-Spencer class.

Moreover, this theorem is truly a higher dimensional geometry result. Indeed, for the fibered surface case, that is $n=1$, the meromorphic vector field is actually holomorphic under the hypotheses of the Main Theorem; this is consistent with the fact that birational smooth curves are in fact isomorphic; see: Remark \ref{dimensione1}.

Finally in Theorem \ref{volumetrico}, we obtain the global version of the Volumetric Theorem,  \cite[Theorem 1.5.3]{PZ}, as an immediate corollary of the Main Theorem; then many applications follow, see for example those contained in \cite{CRZ}.
 
The same arguments of the Main Theorem also apply  in some classical cases; in Subsection \ref{bogo} we consider the case of Bogomolov sheaves, see \cite{B,Cam,V}.

The paper is structured as follows. After giving the necessary notation and setting in Section 2, in Section 3 we recall the definition of deformation supported on a divisor and some immediate key properties. In Section 4 we show that if the general fibers of a fibration are birational then 
their Kodaira-Spencer class is supported on a divisor, while Section 5 is devoted to the vice versa of this result, that is the Main Theorem above.
Finally, Section 6 contains local and global examples showing the necessity of the assumptions of the Main Theorem. This section contains also the aforementioned global version of the Volumetric Theorem.

\begin{ackn}
	The first author has been supported by the Institute for Basic Science (IBS-R003-D1). The second author has been supported by the grant DIMA Geometry PRID\-ZUCC, by PRIN 2017 Prot. 2017JTLHJR \lq\lq Geometric, algebraic and analytic methods in arithmetics\rq\rq and by the project DM737 RIC COLLAB ZUCCONI \lq\lq Birational geometry, Fano manifolds and torus actions\rq\rq CUP G25F21003390007, European Union funds, NextGenerationEU.
	Both authors thank the INdAM Advanced grant Hyperbolicity in Diophantine Geometry directed by Prof. Corvaja for its support.
	
	The authors are members INdAM-GNSAGA.
\end{ackn}

\section{Notation and Setting}
\label{sez1}
Let $X$ be an $n+1$-dimensional smooth complex projective variety and $B$ a smooth complex projective curve.
By $f\colon X\to B$ we {always} mean a semistable fibration, that is a proper surjective morphism with connected fibers $X_b:=f^{-1}(b)$, $b\in B$, and such that the possible singular fibers are reduced and normal crossing, see \cite{Il}. 
 We remark that by the semistable reduction theorem, see cf. \cite{KKMSD,CD}, we can always reduce a fibration to this case up to a sequence of blow-ups of $X$ and cyclic Galois coverings of the base. We denote by $E\subset B$ the divisor of singular values of $f\colon X\to B$ and by $W$ the (support of the) inverse image $f^{-1}(E)$.
 \subsection{The Global Kodaira-Spencer morphism}

We recall some useful notions; see: \cite{RZ4}. The short exact sequence 
\begin{equation}
\label{seq}
0\to f^*\omega_B\to \Omega^1_X\to \Omega^1_{X/B}\to 0
\end{equation} defines the sheaf of holomorphic {relative differentials} $\Omega^1_{X/B}$. This sheaf in general is not locally free and for this reason it is often convenient to consider the logarithmic version of (\ref{seq}):
\begin{equation}
\label{seqlog}
0\to f^*\omega_B(E)\to \Omega^1_X(\log W)\to \Omega^1_{X/B}(\log)\to 0.
\end{equation}
We briefly recall that if $W$ is locally given by $z_1z_2\cdots z_k=0$ in a suitable local coordinate system, the sheaf $\Omega^1_X(\log W)$ of {logarithmic differentials} is the locally free $\sO_X$-module generated by $dz_1/z_1,\ldots,dz_k/z_k,dz_{k+1},\ldots,dz_{n+1}$; see: \cite{De}. The $p$-wedge product of $\Omega^1_{X/B}(\log)$ is denoted by $\Omega^p_{X/B}(\log)$, $p=1,\dots,n$. 

In the case of a semistable fibration, (\ref{seqlog}) is an exact sequence of vector bundles.
It is easy to see that there is an inclusion $\Omega^1_{X/B}\hookrightarrow \Omega^1_{X/B}(\log)$, so that  $\Omega^1_{X/B}$ is at least torsion free in our setting; see: \cite[Section 2]{RZ4}.

Applying the functor $f_*\hhom(-,\sO_X)$ to Sequence (\ref{seq}) we obtain a morphism
\begin{equation}
	\label{ksf}
	T_B\to  \ext^1_f(\Omega^1_{X/B},\sO_X)
\end{equation} which is called the Global Kodaira-Spencer morphism. For the properties of the functor $\ext^1_f(-,-)$ we refer to \cite{Bir}, here we only recall that  
$$
\ext^1_f(\Omega^1_{X/B},\sO_X)\otimes \mC(b)=\textnormal{Ext}^1(\Omega^1_{X_b},\sO_{X_b})=H^1(X_b,T_{X_b})
$$ for a general point $b\in B$. 

The evaluation of (\ref{ksf}) on the general $b\in B$ gives the usual Kodaira-Spencer map 
$$
T_{B,b}\to  H^1(X_b, T_{X_b}).
$$ See for example \cite{Se}.

 Denote by $\Delta\subset B$ an open subset with local coordinate $t$, then the Global Kodaira-Spencer map associates to $\frac{\partial}{\partial t}\in \Gamma(\Delta, T_B)$ a section $\xi\in \Gamma(\Delta, \ext^1_f(\Omega^1_{X/B},\sO_X))$
which, at a general point $b\in \Delta$, recovers the  Kodaira-Spencer class $\xi_b\in H^1(X_b, T_{X_b})$.

 \subsection{The local system of de Rham closed forms}
 \label{The local system}
Denote by $\Omega^1_{X,d}$  the sheaf of de Rham closed holomorphic $1$-forms on $X$, hence $\Omega^1_{X,d}\subset\Omega^1_{X}(\log W)$ and consider the composition
$
f_*\Omega^1_{X,d}\hookrightarrow f_*\Omega^1_{X}(\log W)\to f_*\Omega^1_{X/B}(\log).
$
By \cite[Lemma 2.6]{RZ4} the image $\mD^1$ of the sheaf morphism $f_*\Omega_{X,d}^{1}\to f_*\Omega^1_{X/B}(\log)$ is a local system. 
Equivalently, $\mD^1$ can be seen as the  sheaf defined by the following short exact sequence: 
\begin{equation}
	\label{dx}
	0\to \omega_B\to f_*\Omega^1_{X,d}\to\mD^1\to 0.
\end{equation}

\section{Supported deformations}
We are interested in the case where the image of the Kodaira-Spencer map is a supported infinitesimal (first order) deformation.

\label{sez3}
\subsection{Infinitesimal deformations and divisors}
We recall what it means for $\zeta\in H^1(Z,T_Z)$ to be supported on a divisor $D$ of 
a smooth complex projective variety $Z$.
\begin{defn}
\label{supportato}
We say that the infinitesimal deformation $\zeta\in H^1(Z,T_Z)$ is supported on an effective divisor $D$ if $\zeta$ is in the kernel of the natural map
$$
H^1(Z,T_Z)\to H^1(Z,T_Z(D)).
$$
\end{defn}
\begin{rmk}\label{remarksplit} By the functorial properties of the $\text{Ext}$ functor, since $H^1(Z,T_Z)\cong \text{Ext}^1(\Omega^1_Z,\sO_Z)\cong \text{Ext}^1(\omega_Z,\Omega^{n-1}_Z)$, $\zeta\in H^1(Z,T_Z)$ is supported on $D$ if and only if it is element of the kernel of the homomorphism
	$$
	\text{Ext}^1(\Omega^1_Z,\sO_Z)\to \text{Ext}^1(\omega_Z(-D),\Omega^{n-1}_Z).
	$$
	This means that $\zeta$ is associated to an extension
	$$
	0\to \sO_{Z}\to \sE\to \Omega^1_{Z}\to 0
	$$ that has the following top wedge exact sequence
$$
		0\to \Omega^{n-1}_{Z}\to \hat{\sE}\to \omega_{Z}\to 0,
$$ 
and  the corresponding sequence in $\text{Ext}^1(\omega_Z(-D),\Omega^{n-1}_Z)$ splits:
$$
\xymatrix{
0\ar[r]& \Omega^{n-1}_{Z}\ar[r]& \hat{\sE}'\ar[r]& \omega_{Z}(-D)\ar[r] \ar@/_1.3pc/[l]& 0.
}
$$
\end{rmk}

\subsection{Local deformations and divisors}
%
%
Let $\xi\in \Gamma(\Delta, \ext^1_f(\Omega^1_{X/B},\sO_X))$ be the image of $\frac{\partial}{\partial t}\in \Gamma(\Delta, T_B)$ via  the Global Kodaira-Spencer morphism as in Section \ref{sez1}.
Consider $\sD$ a (horizontal) divisor in $X$ and denote by $D_b$ the restriction of $\sD$ to $X_b$.
\begin{defn}	\label{supportato3}
	We say that $\xi\in \Gamma(\Delta, \ext^1_f(\Omega^1_{X/B},\sO_X))$ is supported on a divisor  $\sD$ in $X$ if 
	\begin{equation}
		\xi\in \Ker (\Gamma(\Delta,\ext^1_f(\Omega^1_{X/B},\sO_X))\to \Gamma(\Delta,\ext^{1}_f(\Omega^1_{X/B}(-\sD),\sO_X))).
	\end{equation}
\end{defn}
\begin{rmk}\label{generalita}
Note that if $\xi$ is supported on $\sD$ then  $\xi_b$ is supported on $D_b$ for the general $b\in \Delta$. The vice versa does not hold, since the sheaf $\ext^{1}_f(\Omega^1_{X/B}(-\sD),\sO_X)$ in general has a torsion part. Nevertheless, since in this paper we are interested in results on the general fibers, we will often assume that $\Delta$ is contained in the Zariski open subset of $B$ where $\ext^{1}_f(\Omega^1_{X/B}(-\sD),\sO_X)$ is locally free and the two conditions are equivalent.
\end{rmk}

The following proposition is important for the Main Theorem of this paper.
\begin{prop}\label{vettorepoli}
	Assume that $\xi\in \Gamma(\Delta, \ext^1_f(\Omega^1_{X/B},\sO_X))$ is supported on $\sD$, then, up to shrinking $\Delta$, there exists a meromorphic vector field $v$ in $f^{-1}(\Delta)$ with poles on $\sD$ such that $df(v)=\frac{\partial}{\partial t}$ where defined.
\end{prop}
\begin{proof}
	Sequence  (\ref{seq})  together with its tensor by $\sO_X(-\sD)$ fits into the commutative diagram 
	\begin{equation}\label{dire}
		\xymatrix{0\ar[r]&f^*\omega_B\ar[r]\ar@{=}[d]&\Omega^1_X\ar[r]&\Omega^1_{X/B}\ar[r]&0\\
			0\ar[r]&f^*\omega_B\ar[r]&\sE\ar[r]\ar[u]&\Omega^1_{X/B}(-\sD)\ar@{=}[d]\ar[r]\ar[u]&0\\
			0\ar[r]&f^*\omega_B(-\sD)\ar[r]\ar[u]&\Omega^1_X(-\sD)\ar[r]\ar[u]&\Omega^1_{X/B}(-\sD)\ar[r]&0}
	\end{equation}
	Applying the functor $f_*\hhom(-,\sO_X)$ to Diagram \ref{dire} we obtain the exact sub-diagram 
	\begin{equation}\label{sottodiagramma}
		\xymatrix{f_*T_{X}\ar[r]\ar[d]&T_B\ar@{=}[d]\ar[r]&\ext^1_f(\Omega^1_{X/B},f^*\omega_B)\ar[d]\\
			f_*\sE^\vee\ar[r]&T_B\ar[r]&\ext^1_f(\Omega^1_{X/B}(-\sD),f^*\omega_B)\\
		}
	\end{equation} 
	Since $\xi$ is supported on $\sD$, Diagram \ref{sottodiagramma} gives the correspondence of local sections over $\Delta$
		\begin{equation*}
		\xymatrix{&\frac{\partial}{\partial t}\ar@{=}[d]\ar@{|->}[r]&\xi\ar@{|->}[d]\\
			v\ar@{|->}[r]&\frac{\partial}{\partial t}\ar@{|->}[r]&0\\
		}
	\end{equation*} and so, up to shrinking $\Delta$, the section $v$.
	
	Finally note that $\sE^\vee\hookrightarrow T_X(\sD)$, hence $v\in \Gamma(\Delta,f_*\sE^\vee)$ can be seen as a vector field in $f^{-1}(\Delta)$ with poles on $\sD$.
\end{proof}
\begin{rmk}
	\label{generalita2}
	By Remark \ref{generalita}, the assumption on $\xi$ in Proposition \ref{vettorepoli} can be replaced with the assumption that $\xi_b$ is supported on $D_b$ for general $b\in \Delta$.
\end{rmk}
\begin{rmk}\label{vlocale}
On an open subset where $f$ is a submersion, we  take local coordinates $t,x_1,\dots,x_n$, so that $f$ is given by the projection on the first coordinate. We also assume that $\sD$ has local equation $\mathfrak{d}=0$.  Then the vector field $v$ given by Proposition \ref{vettorepoli} has local expression
$$
	v=\frac{\partial}{\partial t}+\frac{1}{\mathfrak{d}}(\sum_{i=1}^n a_i\frac{\partial}{\partial x_i})
$$ where $a_i$ are local holomorphic functions. The poles on $\sD$ occur only on the vertical part of the vector field.
\end{rmk}

\section{Birational fibers implies supported deformations}\label{sez4}
In this section, we show that if the general fibers of a fibration are birational then the general Kodaira-Spencer class is supported on a divisor, and this divisor is not movable.
\begin{thm}\label{birimplicasupp}
Let $f\colon X\to B$ be a fibration with birational general fibers. Then the Kodaira-Spencer class $\xi_b$ is supported on a divisor $D_b$ for general $b\in B$. Furthermore $h^0(X_b,\sO_{X_b}(D_b))=1$.
\end{thm}
\begin{proof}
	Consider $f^0\colon X^0\to B^0$ the restriction of $f$ over the Zariski open $B^0\subset B$ where the fibers are all birational and $f$ is smooth. By 
	\cite[Theorem 1.1]{BBG}, up to a finite cover $B'\to B^0$, we have that $X':=X^0\times_{B^0} B'$ 	is birational to the product $X_0\times B'$ where $X_0$ is a projective model
	of the fibers.
	
	Since $B^0$ is a dense open subset of $B$ and taking finite cover does not change the deformation of the general fibers, from now on we work on $X'$, and we still denote a general fiber of $f'\colon X'\to B'$ by $X_b$, and by $\xi_b$ its Kodaira-Spencer class.
	
	We resolve the rational map $X'\dashrightarrow X_0$ 
	by a morphisms $\pi$ to obtain a morphism $\rho\colon \widetilde{X}\to X_0$: 
	$$
	\xymatrix{
	\widetilde{X}\ar^\pi[d]\ar^\rho[rd]\\
	X'\ar@{-->}[r]&X_0.
	}
	$$
	
Let $\tilde{f}:=f'\circ\pi$ and denote by $\widetilde{X}_b$ its fiber over $b\in B$. The general Kodaira-Spencer class $\tilde{\xi}_b\in H^1(\widetilde{X}_b,T_{\widetilde{X}_b})\cong\text{Ext}^1(\Omega^1_{\widetilde{X}_b},\sO_{\widetilde{X}_b})$ is associated to the exact sequence
	$$
	0\to \sO_{\widetilde{X}_b}\to \Omega^1_{\widetilde{X}|\widetilde{X}_b}\to \Omega^1_{\widetilde{X}_b}\to 0.
	$$ 
	The double dual $(\rho^*(\omega_{X_0})_{|\widetilde{X}_b})^{\vee\vee}$ is a line bundle contained in $\Omega^{n}_{\widetilde{X}|\widetilde{X}_b}$. It maps isomorphically onto $\omega_{\widetilde{X}_b}(-\widetilde{D}_b)$ where $\widetilde{D}_b$ is a certain effective divisor in $\widetilde{X}_b$. By Remark \ref{remarksplit}, this shows that $\tilde{\xi}_b$ is supported on $\widetilde{D}_b$. 
	
	To prove the theorem we need to show that ${\xi}_b$ is also supported on a divisor. 
	
	Denoting by $\pi_b\colon \widetilde{X}_b\to X_b$ the restriction of $\pi$, we note that $\pi_b$ is birational for general $b$ and that $\pi_{b*}\sO_{\widetilde{X}_b}(\widetilde{D}_b)=\sO_{X_b}(D_b)$ for a divisor $D_b$ on $X_b$ with $\widetilde{D}_b-\pi_b^*D_b$ exceptional.
	
	Since $\tilde{\xi}_b$ is supported on $\widetilde{D}_b$, its image through the composition 
	$$
	H^1(\widetilde{X}_b,T_{\widetilde{X}_b})\to H^1(\widetilde{X}_b,T_{\widetilde{X}_b}(\widetilde{D}_b))\to H^1(\widetilde{X}_b,(\pi_{b}^*T_{{X}_b})(\widetilde{D}_b))
	$$
	is zero. Hence, by compatibility, the image of ${\xi}_b$ via 
	$$
	H^1({X}_b,T_{{X}_b})\to H^1({X}_b,T_{{X}_b}(D_b))
	$$
	is also zero since the homomorphism
	 $H^1({X}_b,T_{{X}_b}(D_b))\to H^1(\widetilde{X}_b,(\pi_{b}^*T_{{X}_b})(\widetilde{D}_b))$ is injective by the Leray spectral sequence.
	 
	This proves that ${\xi}_b$ is supported on $D_b$, it remains to show that $h^0(X_b,\sO_{X_b}(D_b))=1$. From the above discussion, $\widetilde{D}_b$ is an exceptional divisor for the birational morphism $\widetilde{X}_b\to X_0$ hence $h^0(X_b,\sO_{X_b}(D_b))=h^0(\widetilde{X}_b,\sO_{\widetilde{X}_b}(\widetilde{D}_b))=1$.
\end{proof}
\begin{rmk}\label{necessarie}
From the above theorem and its proof, we note that the supporting divisor satisfies two important conditions. First, it is fiberwise non movable, and second, its image  via the map $X'\dashrightarrow X_0$ is contained in  a proper closed subset of $X_0$. 
In the next section we will prove an inverse of Theorem \ref{birimplicasupp}, so it is natural to expect that these conditions will appear in some form as necessary hypothesis of this inverse.
\end{rmk}

\section{ Supported deformations implies birational fibers}\label{sez5}
Theorem \ref{birimplicasupp} shows that if the general fibers of a fibration are birational, then the general Kodaira-Spencer class $\xi_b$ is supported on a divisor. In this section we study the inverse problem.

\subsection{ Volume detecting subspaces}
Recall that $\mD^1$ is the local system defined by the de Rham closed 1-forms of $X$ as introduced in Subsection \ref{The local system}. 

\begin{defn}\label{volumedet}
We say that a $k+1$-dimensional subspace $U=\langle\eta_1,\dots,\eta_{k+1}\rangle\subset H^0(B,\mD^1)$, $0<k\leq n$, is \emph{volume detecting} if 
\begin{enumerate}
\item  $\eta_1\wedge\dots\wedge\eta_{k+1}=0$ as an element of $f_*\Omega^{k+1}_{X/B}(\log)$ 
\item  the morphism
\begin{equation}\label{strictk}
\bigwedge^{k}U\otimes\sO_B\to f_*\Omega^{k}_{X/B}(\log)
\end{equation}is an injection of sheaves. 
\end{enumerate}
\end{defn} 
We stress that the above second condition is the relative version of the so-called strictness, see: \cite{Ca2}, so we call it in  the same way. 

If $U$ is volume detecting, the relative $k$-forms $\omega_i:=\eta_1\wedge\dots\wedge\hat{\eta_i}\wedge\dots\wedge\eta_{k+1}$, $i=1,\dots,k+1$, generate a sheaf of generic rank one in $\Omega^{k}_{X/B}(\log)$.
We call $\widehat{\sD}$ the divisor given by the horizontal components of the common zeroes of the global sections $\omega_i$ and $\sD$ an effective subdivisor of $\widehat{\sD}$. As before, $D_b$ will denote the restriction of $\sD$ to the fiber $X_b$.

\begin{rmk}\label{rmkkforme}
By taking the tensor of (\ref{strictk}) by the canonical bundle of $B$ we have the injective morphism
	$$
	\bigwedge^{k}U\otimes\omega_B\to f_*\Omega^{k}_{X/B}(\log)\otimes \omega_B.
	$$ 
	Furthermore, note that on $X$ the sheaf $\Omega^{k}_{X/B}(\log)\otimes f^*\omega_B$ injects in  $\Omega^{k+1}_X$ and the quotient is  torsion free by an explicit local computation. In particular 
	$$
	 f_*\Omega^{k}_{X/B}(\log)\otimes \omega_B\to f_*\Omega^{k+1}_X
	$$  is an injection of vector bundles and
	a global $k+1$ form on $X$ which vanishes on the general fiber is a global section of  $\Omega^{k}_{X/B}(\log)\otimes f^*\omega_B$. 
\end{rmk}

As recalled at the beginning of this section, we now consider fibrations such that the general Kodaira-Spencer class is supported on a non movable divisor and study the birationality of the fibers. Hence from now on we denote by $B^0$ the Zariski open subset of $B$ where the fibration $f$ is smooth, the morphism (\ref{strictk}) is an injection of vector bundles, and the Kodaira-Spencer class is supported on a non movable divisor at every point. We denote $X^0:=f^{-1}(B^0)$.

\begin{prop}\label{wedgezerok}
Assume that there exists a  volume detecting subspace $U$ defining a divisor $\widehat{\sD}$ as above. If, for a certain $\sD\subseteq \widehat{\sD}$, $\xi_b$ is supported on $D_b$ and  $h^0(X_b,\sO_{X_b}(D_b))=1$, for general $b\in B$, then there exist closed forms $s_1,\dots,s_{k+1}\in H^0(X^0,\Omega^1_{X^0})$ liftings of respectively $\eta_1,\dots,\eta_{k+1}$ such that
		$s_1\wedge\dots\wedge s_{k+1}=0$.
\end{prop}
\begin{proof}
	By \cite[Lemma 2.2]{RZ4}, Sequence (\ref{dx}) splits, hence we can find $s_1,\dots,s_{k+1}\in H^0(X,\Omega^1_X)$ lifting of $\eta_1,\dots,\eta_{k+1}$ respectively. Now consider $s_1\wedge\dots\wedge s_{k+1}\in H^0(X,\Omega^{k+1}_X)$ and note that by the hypothesis $\eta_1\wedge\dots\wedge\eta_{k+1}=0$ and by Remark \ref{rmkkforme} we actually have that $s_1\wedge\dots\wedge s_{k+1}\in H^0(X,\Omega^{k}_{X/B}(\log)\otimes f^*\omega_B)$. 
	
	By the proof of Proposition \ref{vettorepoli} we can consider the morphism $f_*\sE^\vee\to T_B$ that now we denote by $\alpha$. It fits into the following commutative diagram:
	\begin{equation}\label{diag}
		\xymatrix{
			f_*\sE^\vee\ar^{\alpha}[r]\ar^-{\alpha'}[d]&T_B\ar^{\beta}[d]\\
			\bigwedge^{k}U\otimes f_*\sO_X(\sD)\ar^-{\beta'}[r]&f_*\Omega^k_{X/B}(\log)	}
	\end{equation} 
	The morphism $\beta$ is given by $s_1\wedge\dots\wedge s_{k+1}\in H^0(X,\Omega^{k}_{X/B}(\log)\otimes f^*\omega_B)=H^0(B,f_*(\Omega^{k}_{X/B}(\log))\otimes \omega_B)$.
	The vertical arrow $\alpha'$ is given on a section $\theta$ of $f_*\sE^\vee$ by
	\[
	\theta\mapsto \sum_i (-1)^i  \eta_1\wedge\dots\wedge\hat{\eta_i}\wedge\dots\wedge\eta_{k+1}\otimes s_i(\theta)
	\] where $s_i(\theta)$ is the standard contraction; note that $\theta$ is a vector field on $X$ with poles on $\sD$ and $s_i$ is a 1-form on $X$.
	The  horizontal arrow $\beta'$ is given by  $\eta_1\wedge\dots\wedge\hat{\eta_i}\wedge\dots\wedge\eta_{k+1}\mapsto \omega_i$ and 
	 by the fact that  $\sD$ is a divisor of common zeroes of the $\omega_i$. 
	 
	 Now take a small disk $\Delta$ in the open part $B^0$ and take the vector field $v\in \Gamma(\Delta,f_*\sE^\vee)$ given by Proposition \ref{vettorepoli}. Since
	 $\beta(\alpha(v))=\beta'(\alpha'(v))$ and since $h^0(X_b,\sO_{X_b}(D_b))=1$,  we immediately get, in $f^{-1}(\Delta)$, 
	 \begin{equation}\label{localwedge}
	 s_1\wedge\dots\wedge s_{k+1}=\sum \omega_i\wedge f^*\sigma_i
	 \end{equation}
	 with $\sigma_i=s_i(v)dt\in \Gamma(\Delta,\omega_B)$.

	 We can apply the above argument on an open cover of $B^0$ and by Remark \ref{rmkkforme}, the relation (\ref{localwedge}) can be globalised on $X^0$, i.e. we can write:
	 $$
s_1\wedge\dots\wedge s_{k+1}=\sum \omega_i\wedge f^*\Sigma_i
$$	 where $\Sigma_i\in H^0(B^0,\omega_{B^0})$ (eventually zero).

Finally, for every $i=1,...,k+1$ we change liftings of $\eta_i$ as $\tilde{s}_i:=s_i-f^*\Sigma_i$ to get 
$$
\tilde{s}_1\wedge\dots\wedge\tilde{s}_{k+1}=s_1\wedge\dots\wedge s_{k+1}-\sum \omega_i\wedge f^*\Sigma_i=0.
$$

The forms $\tilde{s}_i$ are defined on $X^0$ and are clearly de Rham closed.
\end{proof}

\begin{rmk}
	It is easy to see that from the strictness hypothesis (\ref{strictk}), the 1-forms $s_i$ provided by the previous result are in fact unique.
\end{rmk}

\begin{prop}\label{esisteY}
	Under the same hypotheses of Proposition \ref{wedgezerok}, there is a surjective morphism $h\colon X\to Y$ over a normal $k$-dimensional variety $Y$ of Albanese general type.
\end{prop}
\begin{proof}
By the fact that they are de Rham closed, the forms $s_1,\dots,s_{k+1}\in H^0(X^0,\Omega^1_{X^0})$ of Proposition \ref{wedgezerok} define a foliation $\sF$ on $X^0$ of rank $n+1-k$. The field of rational functions constant on the leaves of $\sF$ is contained in the field of rational functions of $X$, $\mC(\sF)\subset \mC(X)$, hence it gives a rational dominant map $X\dashrightarrow Y'$.

A standard local argument, see for example the proof of \cite[Theorem 1.14]{Ca2}, shows that the forms $s_i$ are pullback of meromorphic forms $\sigma_i$ on $Y'$, so in particular $\dim Y'=k$. The restriction $X_b\dashrightarrow Y'$ to the general fiber is again dominant by the strictness hypothesis and, since the forms $s_i$ are holomorphic when restricted to the general $X_b$, the $\sigma_i$ must be holomorphic on $Y'$, since the pullback of a form with nontrivial poles via a dominant map has nontrivial poles.

Hence the $s_i$ are actually elements of $H^0(X,\Omega^1_{X})$ and a straightforward application of the Generalised version of the Castelnuovo-de Franchis Theorem, \cite[Theorem 1.14]{Ca2},\cite[Proposition II.1]{R}, \cite{Mok}, gives the morphism $h\colon X\to Y$.
\end{proof}
	
The proof of Proposition \ref{wedgezerok} gives the following result that we collect in a separate Lemma for future reference.
\begin{lem}\label{vstanelker}
	In the same setting of Proposition \ref{wedgezerok}, 
	let $s_1,\dots,s_{k+1}\in H^0(X^0,\Omega^1_{X^0})$ be the  (unique) closed forms such that $s_1\wedge\dots\wedge s_{k+1}=0$ and let $v\in \Gamma(f^{-1}(\Delta),\sE^\vee)$, $\Delta\subset B^0$, be a vector field with poles on $\sD$ 
	as in Proposition \ref{vettorepoli}.
	Then $s_i(v)=0$ in $f^{-1}(\Delta)$, $i=1,\dots,n+1$.
\end{lem}
\begin{proof}
	We use Diagram (\ref{diag}) over $\Delta$. Then $\beta\equiv 0$ since $s_1\wedge\dots\wedge s_{k+1}=0$.
	On the other hand note that on $\Delta\subset B^0$, $\beta'$ is an injection of vector bundles by the hypothesis of strictness of $U$ and by the fact that $f_*\sO_X(\sD)$ is a line bundle on $\Delta$.
	It  follows that
	\[
	\sum_i (-1)^i  \eta_1\wedge\dots\wedge\hat{\eta_i}\wedge\dots\wedge\eta_{k+1}\otimes s_i(v)=0
	\] as section of $\bigwedge^{k}U\otimes f_*\sO_X(\sD)$, hence $s_i(v)=0$ for $i=1,\dots,n+1$.
\end{proof}
\subsection{The two foliations}

To understand our main result, we need to introduce two  foliations. We refer for example to \cite{D} for all the basic definitions on foliations.

As in Proposition \ref{esisteY}, we denote by $\sF\subset T_{X}$ the foliation given by the kernel of the 1-forms $s_i$ of Proposition \ref{wedgezerok}. 
It is easy to see that $\sF$ is the foliation associated to the morphism $h\colon X\to Y$ of Proposition \ref{esisteY}. 
In other words $\sF$ is the relative tangent $T_{X/Y}$ of $h$. In particular $\sF$ has $\rank \sF=n+1-k$. 

Consider now the vector field with poles $v$ of Proposition \ref{vettorepoli}. It  can be expressed locally as $v=\frac{1}{\mathfrak d}\hat{v}$ where $\mathfrak d=0$ is a local equation of $\sD$. The local holomorphic vector fields $\hat{v}$ define (up to saturation) a rank one foliation on $f^{-1}(\Delta)$. We denote this foliation by $\sF_v$. We stress that the codimension 2 subset $\Sigma_v$ where $\sF_v$ is not regular is contained in the divisor $\sD$; see the local expression of $v$ in Remark \ref{vlocale}. 

Lemma \ref{vstanelker} can be reformulated  by the following
\begin{cor} \label{vstanelker2}
Let $\Delta\subset B^0$, then $\sF_v\subset \sF$ in $f^{-1}(\Delta)$.
\end{cor}

\subsection{Main Result}
Consider a proper holomorphic submersion over a  disk $\sX\to\Delta$. It is well known, see \cite{M}, that the existence of a holomorphic vector field $v$ on $\sX$ which lifts the vector field $\partial/\partial t$ on $\Delta$ shows that the family is locally trivial. The key idea is that around each point of the central fiber, it is possible to locally find the flow of $v$. By the fact that the fibers are compact, this process is finite, hence, up to shrinking the base $\Delta$, the flow of $v$ gives the desired trivialisation.

In our setting, we have introduced in Proposition \ref{vettorepoli} a vector field $v$ on $f^{-1}(\Delta)$ with poles on a divisor $\sD$. Since $v$ can be seen as a  holomorphic vector field in $f^{-1}(\Delta)\setminus \sD$, it is natural to ask if a similar argument proves that the fibers of $f$ over $\Delta$ are biholomorphic outside a divisor. The problem is that if we remove a divisor the fibers are no longer compact, hence the above argument fails in general since it may not be possible to define the flow over a  fixed $\Delta$ of strictly positive radius. We give examples of this in the last Section. Nevertheless, the notion of volume detecting subspace together with a natural tangency condition give the vice versa of Theorem \ref{birimplicasupp}. 

We recall that 
	\begin{defn}\label{invariante} A subvariety $W\subseteq X$ is invariant under a foliation $\sF$ if for every local section $\partial$ of $\sF$ over some
		open subset $U$ of $X$, $\partial(I_{W\cap U})\subset I_{W\cap U}$ where $I_{W\cap U}$ is the defining ideal of $W$ in $U$.
	\end{defn}
	To prove that $W$ is invariant under $\sF$ it is enough to show
	that $W\cap U$ is invariant under $\sF|_U$ for some open set $U\subseteq X$ such that $W\cap U$ is dense in $W$.
	See \cite{D}.

\begin{thm}[Main Theorem]\label{main}
	Let $f\colon X\to B$ be a semistable fibration with fibers of general type. Assume that there exist a volume detecting subspace $U$ with associated divisor $\sD$ such that $h^0(X_b,\sO_{X_b}(D_b))=1$ for general $b$. If $\xi_b$ is supported on $D_b$ for general $b$ and $\sD$ is  invariant under $\sF$  then the general fibers of $f\colon X\to B$ are birational.
\end{thm}
\begin{rmk}\label{necessarie2}
	Before starting the proof, we note that the hypothesis on $\sD$, that is its invariance under $\sF$ and the fact that $h^0(X_b,\sO_{X_b}(D_b))=1$ for general $b$, are natural conditions for the birationality of the general fibers as we have seen in Remark \ref{necessarie}.
\end{rmk}
\begin{proof}
	The morphism $h\colon X\to Y$ obtained in Proposition \ref{esisteY} is constant on the leaves of the foliation $\sF$ and the general leaf is of dimension $n+1-k.$ Moreover, since the horizontal divisor $\sD$ is invariant under $\sF$,  ${h(\sD)}$ is at most a divisor in $Y$.
	We denote by $Z\subset Y$ the  proper Zariski closed subset containing $h(\sD)$ and the locus  of singular values of  $h$. 
	
	We fix a disk $\Delta\subset B^0$ where we can apply Proposition \ref{vettorepoli} and take a meromorphic vector field $v$ in $f^{-1}(\Delta)$ with poles on $\sD$ such that $df(v)=\frac{\partial}{\partial t}$.  Let $X_0$ be the central fiber over $\Delta$. We show that the  integral curves of $v$ passing through a general point in $X_0$  (in particular points outside $D_0$) can be defined over the whole disk $\Delta$, independently from the chosen initial point. To do that we restrict $h$ to $X_0$ and we call $h_0\colon X_0\to Y$ this restriction. We choose a initial point $x_0\in X_0\setminus h_0^{-1}(Z)$ and denote by $F_0$ the leaf of the foliation $\sF_v$ passing through $x_0$. 
	
	By Lemma \ref{vstanelker}, cf. Corollary \ref{vstanelker2}, the (Zariski) closure of $F_0$ is contained in the smooth $n+1-k$-dimensional subvariety ${h^{-1}(h(x_0))}\cap X^0$  and by our choice of $x_0$ it does not intersect the divisor $\sD$. In particular this implies that every point of $F_0$ is a regular point of the foliation $\sF_v$, since the singular locus $\Sigma_v$ is contained in $\sD$. Then $F_0$ is locally immersed as a smooth subvariety of $f^{-1}(\Delta)$. Let $g\colon F_0\to \Delta$ be the composition
	\begin{equation*}
		\xymatrix{
		F_0\ar[r]\ar_g[dr]&f^{-1}(\Delta)\ar^f[d]\\
		&\Delta}
	\end{equation*}
	that is the restriction of $f$ to the leaf. It is easy to see that $g$ is a submersion between $1$-dimensional manifolds. 
	
	We claim that $g$ is surjective. Indeed assume by contradiction that $g(F_0)$ is a strict subset of $\Delta$. We know that $g(F_0)$ is open since $g$ is a submersion. Take a point $\bar{t}$ on the boundary of $g(F_0)$ inside $\Delta$. Take a converging sequence $t_n\to \bar{t}$ with $t_n\in g(F_0)$ and choose $x_n\in F_0$ such that $g(x_n)=t_n$. Since $f$ is proper, we can find (up to passing to a subsequence) a limit  point $\bar{x}\in f^{-1}(\Delta)$ such that $f(\bar{x})=\bar{t}$. Note that $\bar{x}$ is in $\overline{F_0}\subset h^{-1}(h(x_0))\cap X^0$ hence $\bar{x}$ is a regular point of $\sF_v$. The foliation $\sF_v$ is transversal to the fiber $f^{-1}(\bar{t})$ in $\bar{x}$ and since there is a sequence of points of $F_0$ converging to $\bar{x}$, it easily follows that there exists a point in $f^{-1}(\bar{t})\cap F_0$, showing that $\bar{t}\in g(F_0)$. 
	
	Since $\Delta$ is simply connected and $F_0$ is connected, $g$ is invertible and its inverse is the parametrisation of the integral curve for the vector field $v$ passing for $x_0$, hence it is defined on the whole disk $\Delta$ as required.
	
	Repeating the argument for every point of $X_0\setminus h_0^{-1}(Z)$ and recalling that the flow of a holomorphic vector field is holomorphic, we get a holomorphic map $X_0\setminus h_0^{-1}(Z)\to X_t$ which is a biholomorphism on the image for every $t\in \Delta$. 
	
	By \cite[Theorem 2]{KO} this map gives a bimeromorphic map $X_0\to X_t$ since $X_t$ is of general type and $h_0^{-1}(Z)$ is analytic. 
	Finally by \cite{Ch}, bimeromorphic projective varieties are birational.
\end{proof}

By \cite[Theorem 1.1]{BBG}, we have the corollary
\begin{cor}
Under the same hypotheses of Theorem \ref{main}, let $X_0$ be a projective model of the fibers. Then there exists a dense open subset $B^0 \subset B$ and  a finite cover $B'\to B^0$ such that $X':=X\times_{B^0} B'$ is birational to $X_0\times B'$.
\end{cor}
\begin{rmk} Since the differential forms $\omega_i=\eta_1\wedge\dots\wedge\hat{\eta_i}\wedge\dots\wedge\eta_{k+1}$, $i=1,...,k+1$ are pullback of forms on $Y$, the divisor on $X_b$ where they all vanish is comprised by components which are pullback of a divisor on $Y$ and by components which are  critical loci of the surjective morphism $h_b:=h_{|X_b}\colon X_b\to Y$, i.e. where the map is not a submersion. In the context of Theorem \ref{main}, the components that matter are the latter. For example, when $k=n$, by the same argument of Theorem \ref{birimplicasupp}, the pullback $(h^*\omega_Y)^{\vee\vee}$ shows that, for general $b$, $\xi_b$ is supported on the divisorial component  of the critical locus of $h_b$. 
\end{rmk}
\begin{rmk}\label{dimensione1}
	Under the hypotheses of Theorem \ref{main}, when $f\colon X\to B$ is a fibered surface (i.e. $n=k=1$), the general fibers are actually isomorphic since they are smooth curves. This is consistent with the following observation which is easy by local computation recalling the local expression of $v$ of Remark \ref{vlocale}: if the vector field $v$ has poles on $\sD$ and, at the same time,  $\sD$ is invariant under $\sF_v$, then $v$ is actually holomorphic. 
\end{rmk}

\begin{rmk}\label{versionemt} 
In the case $k=n$, Theorem \ref{main} can be rephrased by saying that the general fibers of $f\colon X\to B$ are birational if $U$ is strict,  
$U_b$ is Massey trivial for general $b$, and $\sD$ is invariant under $\sF$.
We refer to \cite{PZ,RZ1} for the definition of Massey trivial subspace $U_b$, with the remark that in those papers this notion is expressed saying that the adjoint image of $U_b$ is trivial. 

To show this note that, if $k=n$, the notion of volume detecting is equivalent to asking that $U$ is strict.
 Furthermore, the assumption that $\xi_b$ is supported on $D_b$ is implied by the condition of Massey triviality of $U_b$, see \cite[Theorem A]{RZ1}. 
 
Finally, the condition $h^0(X_b,\sO_{X_b}(D_b))=1$ is not needed. In fact note that the key points where this condition is used is to show the existence of the forms $s_1,\dots,s_{k+1}$ with $s_1\wedge\dots\wedge s_{k+1}=0$ (Proposition \ref{wedgezerok}) and the existence of a meromorphic vector field $v$ in $f^{-1}(\Delta)$ with $df(v)=\frac{\partial}{\partial t}$ and $s_i(v)=0$, $i=1,\dots,n+1$, (Corollary \ref{vstanelker}). If we assume that $U_b$ is Massey trivial for general $b$, the existence of the $s_i$ follows from \cite[Proposition  4.10]{RZ4}, hence we still have the morphism $h\colon X\to Y$ as in Proposition \ref{esisteY}. The existence of the meromorphic vector field follows by choosing $v$ as the (unique) lifting of the vector field $\frac{\partial}{\partial t}$ on $Y\times \Delta$ via the morphism $h\times f\colon f^{-1}(\Delta)\to Y\times \Delta$ which is a local isomorphism outside of its critical locus $\sD$.
 Given these two ingredients, the proof of Theorem \ref{main} proceeds in the exact same way.
\end{rmk}

\subsection{Bogomolov sheaves and birational fibers}
\label{bogo}

In Theorem \ref{main}, the morphism $h\colon X\to Y$ controlling the flow of the vector field $v$ and ensuring the birationality of the fibers is obtained thanks to the volume detecting subspace $U$ and the fact that the general deformations are supported.

There are also classical cases where a similar rational map exists and we can use the same argument of Theorem \ref{main}; for example we recall the theory by Bogomolov, see \cite{B, Cam, V}.
A line bundle (not
necessarily saturated) $\sL\subset\Omega_X^p$ is called a Bogomolov sheaf if its Iitaka dimension is $k(\sL)=p$. A Bogomolov sheaf induces a dominant rational map $h\colon X\dashrightarrow Y$  over a smooth variety $Y$ with $\dim Y=p$ and such that $\sL$ coincides with $h^*\omega_Y$ on a Zariski open set of $X$. It is not true that $Y$ has to be of general type but $h$ has to be of general type in the sense of Campana \cite{Cam}.

Using our relative setting, consider the fibration $f\colon X\to B$ and assume that there exists a Bogomolov sheaf $\sL\subset \Omega^n_X$ and such that the maps of relative differentials $ \Omega^n_X\to \Omega^n_{X/B}$ is an isomorphism when restricted to $\sL$ (i.e. the differential forms of $\sL$ do not come from $B$). In particular the associated map $h$ is dominant when restricted to the general fiber $X_b$. 

Hence we can consider the foliation  $\sF\subset T_X$ induced by $h$ and the divisor $\sD$ of the $f$-horizontal components of the locus where the morphism $h^*{\omega_Y}\otimes f^*\omega_B\to \omega_X$ is not an injection of vector bundles.

We have the following theorem 
	\begin{thm}\label{campanavoisin}
		Let $f\colon X\to B$ be a semistable fibration with fibers of general type.
		Let $\sL\subset \Omega^n_{X}$ be a Bogomolov sheaf such that the morphism $ \Omega^n_X\to \Omega^n_{X/B}$ is an isomorphism when restricted to $\sL$ and such that $Y$ is of general type. If $\sD$ is invariant under $\sF$ then the general fibers of $f\colon X\to B$ are birational.
		\end{thm}
\begin{proof}
		We resolve the rational map $h\colon X\dashrightarrow Y$  via a birational morphism $\pi\colon \widetilde{X}\to X$ to obtain $\tilde{h}\colon \widetilde{X}\to Y$.
	
	For our purpose, it is enough to prove that the general fibers $\widetilde{X}_b$  of $\tilde{f}:=f\circ\pi$ are birational.
	The proof is similar to Theorem \ref{main}, so we only give an idea.

	Given a suitable open subset $\Delta\subset B$, consider the generically finite morphism $\tilde{f}\times\tilde{h}\colon  \tilde{f}^{-1}(\Delta)\to \Delta\times Y$ and take on $\tilde{f}^{-1}(\Delta)$ the meromorphic vector field with poles on $\widetilde{\sD}$ obtained by lifting the vector field $\frac{\partial}{\partial t}$ of $\Delta\times Y$ outside the critical divisor of $\tilde{f}\times\tilde{h}$.
	
 To repeat the proof of Theorem \ref{main}, we only need to show that $\widetilde{\sD}$ is invariant under $T_{\widetilde{X}/Y}$, but this is true since $\sD$ is invariant under $\sF$ and the exceptional components of $\widetilde{\sD}$ cannot cover $Y$ since by hypothesis it is a variety of general type.
\end{proof}

\section{Examples and applications}\label{sez6}
\subsection{Local case}
We present two local examples showing why the assumption of $\sD$ being invariant under $\sF$ (and hence also under $\sF_v$ by Corollary \ref{vstanelker2}) is required  in Theorem \ref{main}. We use the notation of the previous sections. 

\begin{expl}\label{ex1}
	Take an open set of $X$ where the map $f$ is a submersion, in particular we choose holomorphic coordinates $t,x_1,\dots,x_n$ and $f$ is given by the projection on the first coordinate. Furthermore, assume that the divisor $\sD$ is locally given by $x_1=0$ and that 
	$$
	v=\frac{\partial}{\partial t}+\frac{1}{x_1}(\sum_{i=1}^n \frac{\partial}{\partial x_i})
	$$ is the vector field with poles on $\sD$ and lifting $\frac{\partial}{\partial t}$; cf. Remark \ref{vlocale}. Note that in this case 
		$$
	\hat{v}=x_1\frac{\partial}{\partial t}+(\sum_{i=1}^n \frac{\partial}{\partial x_i})
	$$ hence we immediately see that $\sD$ is not invariant under $\sF_v$ as required in Theorem \ref{main}.
	Now choose a starting point on the central fiber $x_0=(0,\lambda_1,\lambda_2,\dots,\lambda_n)$, with $\lambda_1$ in the right half of the complex plane. An easy computation shows that the integral curve passing through this point is 
	$$
	\gamma_{x_0}(t)=(t,\sqrt{2t+\lambda_1^2}, \sqrt{2t+\lambda_1^2}-\lambda_1+\lambda_2,\dots,\sqrt{2t+\lambda_1^2}-\lambda_1+\lambda_n)
	$$ where we denote by $\sqrt{}$ the branch of the square root defined on $\mC\setminus \mR_{<0}$.
	
	For fixed $x_0$ as above, the curve $\gamma_{x_0}$ is defined on a open disk centred in zero and of radius  $\text{dist}(\lambda_1^2,\mR_{<0})/2$, hence it is immediately clear that, moving the point $x_0$  close to the divisor $\sD$, i.e. $\lambda_1\to 0$, these curves are not defined on a common disk $\Delta$ containing the origin and of strictly positive radius independent of the starting point.
\end{expl}
In Example \ref{ex1}, the morphism $h$ and the associated foliation $\sF$ do not play any role; the next example shows that even if $\sD$ is invariant under $\sF_v$, this is not sufficient to apply the argument of the Main Theorem when $\sD$ is not invariant under $\sF$.
\begin{expl}
	For simplicity we study a 3-dimensional case, i.e. $n=2$. Assume again that $f$ is a submersion and take local coordinates $t,x,y$. Again the divisor $\sD$ is given by $x=0$ but now take as vector field 
	$$
	v=\frac{\partial}{\partial t}+\frac{\partial}{\partial x}+\frac{y^2}{x}\frac{\partial}{\partial y}.
	$$ In this case 
	$$
	\hat{v}=x\frac{\partial}{\partial t}+x\frac{\partial}{\partial x}+{y^2}\frac{\partial}{\partial y}
	$$ and the foliation $\sF_v$ is regular outside of $\Sigma_v=(x=y=0)\subset \sD$.
	
	We denote by $\log$ the branch of the logarithm defined on $\mC\setminus\mR_{\leq0}$, and note that the curves
		$$
	\gamma_{x_0}(t)=(t,t+\lambda, -\frac{1}{\log(t+\lambda)+\mu})
	$$ are integral curves passing on the central fiber through points of the form $x_0=(0,\lambda,-\frac{1}{\log(\lambda)+\mu})$ (we assume of course that $\lambda\notin \mR_{\leq0}$ and $\log(\lambda)+\mu\neq0$).
	
	Here the problem is the same of the previous example: moving the point $x_0$  close to the divisor $\sD$, i.e. $\lambda\to 0$, the curves $\gamma_{x_0}$ are not defined on a common disk $\Delta$ containing the origin and of strictly positive radius independent of $x_0$. But the difference with the previous example is that here $\sD$ is invariant under $\sF_v$; the problem in this case arises from the map $h$ and its associated fibration $\sF$. 
	
	First note that if we denote as before by $F_0$ the leaf of $\sF_v$ through $x_0$, the point $(-\lambda,0,0)\in \Sigma_v$ is in the closure of $F_0$ for every $x_0$ as above. Since we recall that $h$ is a morphism constant on the leaves of this foliation, the general fibers of $h$ are of dimension 2: a fiber containing the leaf through $(0,\lambda,-\frac{1}{\log(\lambda)+\mu})$ must contain also all the other leaves passing through $(0,\lambda,-\frac{1}{\log(\lambda)+\mu'})$. Concretely, we can think of $h(t,x,y)=x-t$ and note that the foliation $\sF$ of rank 2 given by $h$ is generated by the vector fields $\frac{\partial}{\partial t}+\frac{\partial}{\partial x}, \frac{\partial}{\partial y}$ and contains the foliation $\sF_v$ as desired. 
	
	The argument of the Main Theorem does not work in this case because $\sD$ is not invariant under $\sF$; as a consequence, as already pointed out, there is a point of $\Sigma_v$ in the closure of the leaves $F_0$ and the proof Theorem \ref{main} strongly relies on this not happening.
\end{expl}
The following figures depict the issue in the above examples and show how they are solved by the assumptions of Theorem \ref{main}.
\begin{figure}[H]
	\includegraphics[scale=.9]{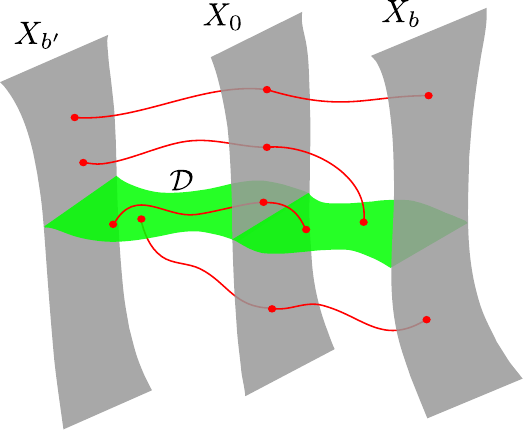}
	\caption{If the divisor $\sD$ is not invariant under the foliation $\sF$, the integral curves of $v$ (in red) \lq\lq meet\rq\rq\ $\sD$ at their boundary. The closer the starting point is to $\sD$, the smaller the set where the integral curve is defined, hence the flow of $v$ does not  give a biholomorphism of the fibers outside the poles, even up to shrinking the base $\Delta$.}
\end{figure}
\begin{figure}[H]
	\includegraphics[scale=.9]{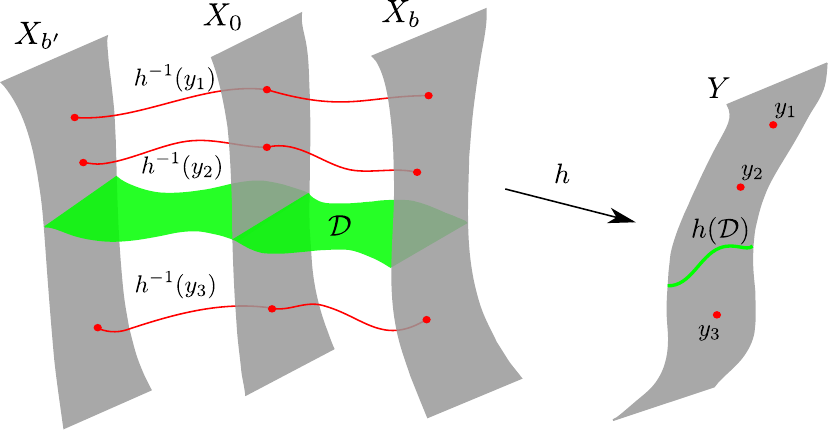}
		\caption{When  the divisor $\sD$ is invariant under the foliation $\sF$, the flow is well behaved. The image of $\sD$ is not dominant in $Y$ and the (general) integral curve is contained in a fiber of $h$ and does not intersect $\sD$.}
\end{figure}
\subsection{Global case}
The following global examples show that the conditions on the global sections of the line bundle $\sO_{X_b}(D_b)$ and of invariance of $\sD$ under the foliation $\sF$ in the main theorem are necessary.  

\begin{expl}\label{doubledouble} Let $j\colon B\hookrightarrow Y$ be a smooth curve embedded into a smooth surface $Y$. To simplify our exposition assume that $Y$ is of general type and without any rational curve. Let $\Gamma\subset B\times Y$ be the image of the diagonal embedding $b\mapsto (b,j(b))$ of $B$ inside $B\times Y$. We denote by $\pi_1\colon B\times Y\to B$ and $\pi_2\colon B\times Y\to Y$ the two natural projections. Let $\sigma\colon T\to B\times Y$ be the blow-up of $B\times Y$ along $\Gamma$. We denote by $\widetilde\Gamma:=\sigma^* \Gamma$ the exceptional divisor and by $S_b:= \sigma^{-1}\circ\pi_1^{-1}(b)$, the fiber over $b\in B$; the morphism $\pi_1\circ\sigma\colon T\to B$ is a smooth fibration with fibers $S_b$. By \cite[Corollary 7.15]{Ha} it follows that letting $\sigma_b:=\sigma_{|S_b}$, $\sigma_b\colon S_b\to Y$ is the blow-up of $Y$ at $j(b)$. We denote by $F_b$ the $-1$-exceptional curve on $S_b$ and we stress that $F_b=\widetilde\Gamma_{| S_b}=\sigma^{-1}(b,j(b))$. 
	
The example is constructed considering a family of ramified double covers of $S_b$ as follows, see \cite{CS}.
 Let $\sM'\sim 2\sL'$ be a very ample $2$-divisible sheaf on $B\times Y$. We consider the divisors $\sM:=\sigma^*\sM'$ and $\sL:=\sigma^*\sL'$ on  $T$. Let $\sG\in |\sM|$ be a smooth divisor and let $\rho\colon X\to T$ be the $2$-to-$1$ cover branched over $\sG$. Let $f\colon X\to B$ the fibration given by $\pi_1\circ\sigma\circ\rho$ and $h\colon X\to Y$ the morphism $\pi_2\circ\sigma\circ\rho$. 
 Note that $\rho$ induces a $2$-to-$1$ cover $\pi_b\colon X_b\to S_b$ of $S_b$ branched over $G_b:=\sG_{|S_b}\in | M_b |$, where $M_b=\sM_{|S_b}$. We recall that in this case $R_b:=\pi_b^{-1}(G_b)$ is the ramification divisor and that $\pi_{b| R_b}\colon R_b\to G_b$ is an isomorphism. We set $L_b:=\sL_{| S_b}$ and $E_b:=\pi_b^{-1}(F_b)$. 
 
 In this example we take as $U$ a suitable vector space of one-forms of $Y$ pulled back via $h$, which in this example plays the role of the morphism $h$ of Proposition \ref{esisteY}. 
 We note that, using the pullback $h^*\omega_Y$ as in Theorem \ref{birimplicasupp}, it easily follows that the image of the Kodaira-Spencer map at $b$, denoted as always by $\xi_b\in H^1(X_b, T_{X_b})$, is supported on $E_b+R_b$, the critical divisors of $h_b$. Note however that $h^0(X_b, \sO_{X_b}(E_b+R_b))>1$ and the components given by the $R_b$ are not invariant under $\sF=T_{X/Y}$. Hence we are not in the hypotheses of Theorem \ref{main} and in fact the fibers $X_b$ are not birational:
 \begin{prop}\label{nonsono} The fibers of $f\colon X\to B$ are not birationally equivalent.
 \end{prop}
 \begin{proof} Let $Z_b$ be the minimal model of $X_b$. By generality, if $Z_b$ were birationally equivalent to $X_{b'}$ then $Z_b$ would be biregular to $Z_{b'}$ and letting $Z:=Z_b$ we would have that $Z$ could be obtained as a double covering of $Y$ through an infinite number of distinct morphisms. In particular $Z$ would be a surface of general type with an infinite group of automorphisms. 
 \end{proof}

\end{expl}
In Example \ref{doubledouble}, we see that both the conditions on the supporting divisor from Theorem \ref{main} are not satisfied. In the next example, the supporting divisor is fiberwise non movable but it is not invariant under the foliation. 

\begin{expl}
	Consider $C$ a smooth  hyperelliptic curve of genus $g$ with hyperelliptic involution $\sigma$ and let $V:=C\times C$ be the self product and $Y:=\text{Sym}^2C$ the symmetric product. We denote by $\Delta_C\subset V$ the diagonal and $\Delta\subset Y$ its image in $Y$ via the natural ramified covering $\pi\colon V\to Y$. We also denote by $\delta$ the class giving the double covering $\pi$, $\Delta\simeq 2\delta$, and by $L:=\pi(C\times \{\Sigma_{i=1}^{g-1} P_i\})\subset Y$ the divisor associated to the sum of $g-1$ distinct Weierstrass points $P_i\in C$.
	
	Similarly to the previous example, we consider a family $f\colon X\to B$ with a map $h\colon X\to Y$ such that the restriction $h_b\colon X_b\to Y$ is a $2:1$ covering branched over a divisor in  $|2L+\Delta|$. 
	
	To do this we need to show that the general divisor in $|2L+\Delta|$ is smooth.
	
	First note that  $|\sO_Y(2L)|$ is the linear system giving the morphism
	$$
	Y\to\mP^{\frac{g(g+1)}{2}-1}
	$$ obtained by the composition of the covering $Y\to \text{Sym}^2\mP^1\simeq\mP^2$ and the Veronese map $\nu_{g-1}\colon \mP^2\to \mP^{\frac{g(g+1)}{2}-1}$. In particular $|\sO_Y(2L)|$ is base point free and $h^0(Y,\sO_Y(2L))=\frac{g(g+1)}{2}$.
	
	Now note that we have the  following exact sequence on $Y$
	$$
	0\to \sO_Y(2L)\to \sO_Y(2L+\Delta)\to \sO_\Delta\to 0.
	$$
	To prove that $2L+\Delta$ has no base points we will show that $h^0(Y,\sO_Y(2L+\Delta))=\frac{g(g+1)}{2}+1$.
	
	We have
	$$
	\pi^*(\sO_{Y}(2L+\Delta))=\sO_V(K_V+2\Delta_C)\quad\text{and}\quad \pi_* \pi^*(\sO_Y(2L+\Delta))=(\sO_Y(2L+\Delta))\oplus (\sO_Y(2L+\delta))
	$$ by projection formula and the fact that $\pi_*\sO_{V}=\sO_Y\oplus \sO_Y(-\delta)$. Since these bundles have the same number of global sections, it is enough to compute $h^0(V,\sO_V(K_V+2\Delta_C))$ and $h^0(Y,\sO_Y(2L+\delta))$.
	\begin{lem}\label{conto}
		$h^0(Y,\sO_Y(2L+\delta))=\frac{g(g-1)}{2}$.
	\end{lem} 
	\begin{proof}
		We have
		$$
		\pi^*(\sO_{Y}(2L+\delta))=\sO_V(K_V+\Delta_C)$$ and $$h^0(V,\sO_V(K_V+\Delta_C))=h^0(C,K_C)^2=g^2.$$ Since
		$$\pi_* \pi^*(\sO_Y(2L+\delta))=(\sO_Y(2L+\delta))\oplus (\sO_Y(2L))
		$$  
		and $h^0(Y,\sO_Y(2L))=\frac{g(g+1)}{2}$ as recalled above, we have $h^0(Y,\sO_Y(2L+\delta))=g^2-\frac{g(g+1)}{2}=\frac{g(g-1)}{2}$.
	\end{proof}
	\begin{lem}
		$h^0(V,\sO_V(K_V+2\Delta_C))=g^2+1.$
	\end{lem} 
	\begin{proof}
		We have the short exact sequence 
		$$
		0\to \sO_V(K_V+\Delta_C)\to \sO_V(K_V+2\Delta_C)\to \sO_{\Delta_C}\to 0.
		$$ Since $h^0(V,\sO_V(K_V+\Delta_C))=g^2$,  we only need to show that not all the global sections of $\sO_V(K_V+2\Delta_C)$ vanish on $\Delta_C$.
		
		To see this, consider the difference map into the Jacobian of $C$, $d\colon V\to J(C)$ and recall that  $K_V+2\Delta_C$ is linearly equivalent to the pullback $d^*(2\theta)$, see for example \cite{W}. 
		Since $d(\Delta_C)$ is the origin of $J(C)$ and not all the elements of $|2\theta|$ vanish on the origin we are done.
	\end{proof}
	
	From these lemmas we have $h^0(Y,\sO_Y(2L+\Delta))=g^2+1-\frac{g(g-1)}{2}=\frac{g(g+1)}{2}+1$ as desired.

	As in Example \ref{doubledouble}, note that the Kodaira-Spencer class $\xi_b$ of the fiber $X_b$ is supported on $R_b:=h_b^{-1}(2L+\Delta)$.
	An easy computation shows that the fibers $X_b$ are of general type, hence with the same argument of Proposition \ref{nonsono}, they are not birational. 
	
	With computations similar to Lemma \ref{conto}, we see that in this case $h^0(X_b,\sO_{X_b}(R_b))=1$, so this is an explicit example where the obstacle to the birationality of the fibers is  the fact that the divisor given by the $R_b$ is not invariant under $T_{X/Y}$. 
\end{expl}

 \subsection{Volumetric theorem}
As final application, we can prove the global version of the so called Volumetric Theorem \cite[Theorem 1.5.3]{PZ} without assuming the smoothness of $f$.

\begin{thm}\label{volumetrico}
	Let $f\colon X\to B$ be a semistable family with a morphism $\Phi\colon X\to A$ to an abelian variety such that,  for general $b\in B$, $\Phi_b:=\Phi_{|X_b}\colon X_b\to A$ has generic degree one and its image generates $A$ as an abelian group. 
	Let $W\subset 
	H^{0}(A,\Omega^{1}_{A})$ be a general $(n+1)$-dimensional subspace 
	and denote by $U\subset H^0(B,\mD^1)$ the image of the pullback $\Phi^*W$ in $\mD^1$. Assume that $U_b$ is Massey trivial for general $b\in B$.
	Then the general fibers of $f\colon X\to B$ are birational.
\end{thm}
\begin{proof}
	This can be proved as a corollary of Theorem \ref{main}, cf. Remark \ref{versionemt}.
	Indeed $U_b$  is strict by the generality assumption of $W$, see \cite[Theorem 1.3.3]{PZ}, hence it is an $n$-volume detecting subspace. 
	
	Using the notation of the previous section,
	note that the divisor $D_b$, given by restriction of $\sD$ to each fiber $X_b$, is given by the exceptional divisors of $h_b$ and possibly the
	pullback of a divisor on $Y$. In particular, it is not difficult to see that $\sD$ is invariant under the foliation $\sF$ by the fact that the divisor where $\Phi_b$ is not an isomorphism is contracted by $\Phi_b$.
\end{proof}

Using techniques similar to those developed for the proof of Theorem \ref{main}, we can also reprove the original Volumetric Theorem.
\begin{thm}\label{volteoremavolume}
	Let $f\colon \sX\to \Delta$ be a smooth local family over a disk with a morphism $\Phi\colon \sX\to A$ to an abelian variety such that,  for general $b\in \Delta$, $\Phi_b:=\Phi_{|X_b}\colon X_b\to A$ has generic degree one and its image generates $A$ as an abelian group. 
	Let $W\subset 
	H^{0}(A,\Omega^{1}_{A})$ be a general $(n+1)$-dimensional subspace 
	and assume $U_b:=\Phi_b^*W$ is Massey trivial for {general} $b\in \Delta$. Then the general fibers of $f\colon \sX\to \Delta$  are birational.
\end{thm}
\begin{proof}
	As in the global case, for {general} $b\in \Delta$, the vector space $U_b$ is strict by the generality assumption of $W$ and Massey trivial by hypothesis. As before denote by $U\subset H^0(B,\mD^1)$ the image of the pullback $\Phi^*W$ in $\mD^1$. Up to shrinking $\Delta$, we can then assume that $U$ is not only fiberwise Massey trivial and strict, but also on the whole disk, according to Definition 4.4 and Definition 4.5 of \cite{RZ4}. By \cite[Proposition 4.7]{RZ4} we can find de Rham closed forms $s_1,\dots,s_{n+1}\in H^0(\sX,\Omega^1_\sX)$ which are liftings of a basis of $U$ and such that
	$s_1\wedge\dots\wedge s_{n+1}=0$.  Note that this is just a local version of Proposition \ref{wedgezerok} with $k=n$.
	As in Proposition \ref{esisteY}, the $s_i$ define a foliation $\sF$ and we note that the field $\mC(\sF)$ is contained in $\mC(X_b)$, $b\in\Delta$. We have then a meromorphic map $\sX\dashrightarrow Y'$ to a $n$-dimensional variety $Y'$ and the pullback of global holomorphic 1-forms on $Y'$ give global holomorphic 1-forms on $\sX$.
	
	The universal covering $\widetilde{\sX}$ of $\sX$ has then a holomorphic map $\widetilde{\sX}\to \mC^l$ where $\mC^l=H^{0}(Y',\Omega^{1}_{Y'})^\vee$. This map is generically of degree one when restricted to the fibers since $\Phi_b$ is.

	 By taking the quotient by the fundamental group, we get a holomorphic map $\sX\to Y$ to a $n$-dimensional variety $Y$ and, by our observation on the degree, all the fibers $X_b$ are actually birational to $Y$, giving the desired result.
\end{proof}

	\begin{rmk}
	A key difference between Theorem \ref{main} and the local Volumetric Theorem \ref{volteoremavolume} is that the former is stated over a compact curve $B$, while the latter over a disk $\Delta$. The reason is that in the proof of Theorem \ref{main} the compactness of $X$ and $B$ ensures that $h(\sD)$ is a proper analytic subset of $Y$; in the local setting the image of a divisor via a holomorphic map is not necessarily analytic.
	
	On the other hand, in Theorem \ref{volteoremavolume}, the argument is easier since the morphisms are of generic degree one  on the fibers and this gives immediately the birationality, without the need to control the image of $\sD$ in $Y$.
\end{rmk}


\end{document}